\documentclass[a4paper,11pt]{amsart}
\usepackage[utf8]{inputenc}
\usepackage[dvips]{graphicx}
\usepackage[shortlabels]{enumitem}
\usepackage{amsmath,amssymb,color,enumitem,graphics,hyperref,latexsym,pgfplots,tikz,vmargin}
\hypersetup{
unicode=false,          
pdftoolbar=true,        
pdfmenubar=true,        
pdffitwindow=false,     
pdfstartview={FitH},    
pdftitle={Sets with large intersection properties in metric spaces},    
pdfauthor={Felipe Negreira, Emiliano Sequeira},     
pdfsubject={11J83, 28A78},   
pdfkeywords={Diophantine approximations, metric spaces, net measures}, 
pdfnewwindow=true,      
colorlinks=true,       
linkcolor=black,          
citecolor=black,        
filecolor=magenta,      
urlcolor=cyan           
}

\def\A{{\mathbf{A}}}
\def\C{{\mathbf{C}}}

\def\diam{{\mathrm{diam}}}
\def\Ee{{\mathcal{E}}}
\def\eps{{\epsilon}}
\def\Ff{{\mathcal{F}}}
\def\Gg{{\mathcal{G}}}
\def\Hh{{\mathcal{H}}}
\def\inte{{\mathrm{int}}}
\def\Jj{{\mathcal{J}}}
\def\lev{{\mathrm{lev}}}
\def\Ll{{\mathcal{L}}}
\def\Mm{{\mathcal{M}}}
\def\N{{\mathbb{N}}}
\def\Qq{{\mathcal{Q}}}
\def\R{{\mathbb{R}}}

\def\Z{{\mathbb{Z}}}
\def\K{{\mathbf{K}}}

\newcommand{\norm}[1]{{\|{#1}\|}}
\newcommand{\abs}[1]{{\left|{#1}\right|}}

\newtheorem{lemma}{Lemma}[section]
\newtheorem{proposition}[lemma]{Proposition}
\newtheorem{theorem}[lemma]{Theorem}

\theoremstyle{definition}
\newtheorem{definition}[lemma]{Definition}
\newtheorem{example}[lemma]{Example}

\theoremstyle{remark}

\begin{document}
\title{Sets with large intersection properties in metric spaces}
\author[F. Negreira, E. Sequeira]{Felipe Negreira, Emiliano Sequeira}
\address{Departamento de Matemática y Estadística del Litoral, Universidad de la República, Uruguay.}
\email{felipenegreira@adinet.com.uy}
\address{Sobolev Institute of Mathematics, 4 Acad. Koptyug Ave., Novosibirsk 630090, Russia}
\email{esequeiramanzino@gmail.com}

\begin{abstract}
In this work we reproduce the characterization of $\Gg^s$-sets from the euclidean setting \cite{F3} to more general metric spaces. These sets have Hausdorff dimension at least $s$ and are closed by countable intersections, which is particularly useful to estimate the dimension of the so called sets of $\alpha$-approximable points (that typically appear in Diophantine approximations).
\end{abstract}
\maketitle

\section{Introduction}
In \cite{F3} Falconer gave a characterization for certain classes of sets in $\R^d$ with Hausdorff dimension at least $s$ that have a ``large intersection'' property in the sense that their countable intersections also have dimension $s$ or larger. These are the so called $\Gg^s${\it-sets}: subsets $F\subset\R^d$ which are $G_\delta$ (countable intersection of open sets) and such that
\begin{equation*}
\dim_\Hh\bigcap_if_i(F)\geqslant s
\end{equation*}
holds for all sequences of similarity transformations $\{f_i\}$, where $\dim_\Hh$ denotes the Hausdorff dimension. Equivalently, one can describe this sets in terms of bi-Lipschitz functions: $F$ is a $\Gg^s$-set if and only if for every open set $V\subset\R^d$ and every sequence of bi-Lipschitz maps $f_i:V\to\R^d$
\begin{equation}\label{gsets}
\dim_\Hh\bigcap_if_i^{-1}(F)\geqslant s.
\end{equation}
Further, the results in \cite{F3} show that $\Gg^s$ is the maximal class of $G_\delta$ sets with Hausdorff dimension $s$ or larger, and which is closed under countable intersections and similarities or bi-Lipschitz homeomorphisms in general.

The main tool used by Falconer to describe $\Gg^s$-sets are the so called {\it net measures} and its respective contents. To construct such measures one needs a suitable family of sets (see \cite[Definition 31]{R}), which in $\R^d$ take the natural form of {\it dyadic partitions}: this is, cubes with the form $Q_{j,k}=2^j(k+[0,1)^d)$, $j\in\Z,k\in\Z^d$. Using these dyadic cubes, the net (outer) measure $\Mm^s$ of a set $E\subset\R^d$ is defined by a Carathéodory construction as $\Mm^s(E):=\lim_{r\downarrow0}\Mm^s_r(E)$, where
\begin{equation}\label{netwithdiam}
\Mm^s_r(E):=\inf\left\{\sum_n\diam(Q_n)^s:E\subset\bigcup_n Q_n, Q_n\text{ dyadic cube with }\diam(Q_n)\leqslant r\right\}.
\end{equation}
In \cite[Theorem B]{F3} it is claimed that \eqref{gsets} is equivalent to
\begin{equation}\label{nottrue}
\Mm^s_\infty(F\cap Q)=\Mm^s_\infty(Q),\quad\text{ for all dyadic cube }Q.
\end{equation}
Actually, this equivalence as stated is not always true: Example \ref{ejemploreal} below shows that \eqref{nottrue} is strictly stronger than \eqref{gsets}. However, as noted by Bugeaud in \cite{B}, this problem can be easily fixed by weakening \eqref{nottrue} to
\begin{equation}\label{Bugeaud}
\Mm^t_\infty(F\cap Q)=\Mm^t_\infty(Q),\quad\text{ for all dyadic cube }Q\text{ and all }t<s.
\end{equation}

When trying to adapt these definitions to other metric spaces, the problem with \eqref{gsets} is that we may have very few similarities or bi-Lipschitz maps and so we might end with a weak condition that fails to preserve some fundamental properties of $\Gg^s$-sets in $\R^d$. Indeed, as Example \ref{ejemplopegado} shows, it is possible to construct one of such metric spaces where there exists a large class of sets verifying \eqref{gsets} but, for example, which is not closed under countable intersections. Instead, provided that we have a dyadic-type partition, \eqref{nottrue} or \eqref{Bugeaud} give a better generalization.

The existence of dyadic-type cubes in more general settings such as {\it spaces of homogeneous type} (quasi-metric spaces with a doubling Borel measure) is known since at least the work of M. Christ \cite{Ch} in 1990. More recently, T. Hytönen and A. Kairema \cite{HK} gave a sharper dyadic decomposition for any {\it geometrically doubling} metric space. One could then expect the results of Falconer \cite{F3} to be also valid in these metric (or quasi-metric) spaces. However, dyadic cubes in an arbitrary metric space as understood in \cite{Ch,HK} are not exactly perfect cubes as in $\R^d$ but rather {\it quasi-balls}: sets trapped in between a ball $B$ and a scalar multiple $\K B$, see $\S$\ref{DescDiadica} below. Thus, in the more general case, we only have an approximate idea of the diameter of a dyadic cube. The problem with this is that, a priory, we do not know whether the function $Q\mapsto\diam(Q)^s$ is sub-additive or not -a property that plays key role in many of the proofs of \cite{F3}.

When we have a measure with enough regularity, a way to circumvent this is to change $\diam(Q)^s$ in \eqref{netwithdiam} for the measure of $Q$ taken with an appropriate exponent. For example, in $\R^d$ for $0<s\leqslant d$ and $0<r\leqslant\infty$ if we set
\begin{equation*}
\Mm^s_{\Ll,r}(E):=\inf\left\{\sum_n\Ll(Q_n)^{s/d}:E\subset\bigcup_n Q_n, Q_n\text{ dyadic cube with }\Ll(Q_n)\leqslant r\right\},
\end{equation*}
where $\Ll$ denotes the usual Lebesgue measure, then $\Mm^s_\Ll=\lim_{r\downarrow0}\Mm^s_{\Ll,r}$ and $\Mm^s_{\Ll,\infty}$ are scalar multiples of $\Mm^s$ and $\Mm^s_\infty$ respectively. Moreover, one could work with more general Borel measures $\mu$ instead of $\Ll$ and still be able to reproduce much of Falconer's results. Indeed, in \cite{P} Persson shows how, keeping the same standard dyadic partition in $\R^d$ but working with an arbitrary non-atomic locally finite Borel measure $\mu$, one can define the equivalent $\Gg^s_\mu$-sets in a way such that many the same properties of \cite{F3} still hold.

Here, using the same notation as in \cite{P}, we translate these definitions into a metric space of homogeneous type.

\begin{definition}\label{jsets}
Let $(X,\mu)$ be a metric space of homogeneous type. Let $d=\dim_\Hh X$, take $\Qq$ a dyadic partition in $X$ and $0<s\leqslant d$. We say that a $G_\delta$ set $F\subset X$ is a $\Jj^s_\mu$-set if
\begin{equation*}
\Mm^s_{\mu,\infty}(F\cap Q)=\Mm^s_{\mu,\infty}(Q)
\end{equation*}
holds for all $Q\in\Qq$. We donote the class of $\Jj^s_\mu$-sets in $X$ by $\Jj^s_\mu(X)$ or simply $\Jj^s_\mu$.

We also define the class of $\Gg^s_\mu$-sets in $X$ by 
\begin{equation*}
\Gg^s_\mu=\Gg^s_\mu(X):=\bigcap_{t<s}\Jj^t_\mu(X).
\end{equation*}
\end{definition}

%

Note that the condition that defines a $\Gg^s_\mu$-set in a general metric space $X$ is effectively the equivalent of \eqref{Bugeaud} in $\R^d$. For both of these type of sets the corresponding equivalences of \cite[Theorem B]{F3} still hold.

\begin{theorem}\label{teoB}
Let $(X,\mu)$ be a metric space of homogeneous type with $d=\dim_\Hh X$. Take $F\subset X$ a $G_\delta$ subset and $0<s\leqslant d$. Then the following statements are equivalent:
\begin{enumerate}[(i)]
\item $F$ is a $\Jj^s_\mu$-set, i.e. for all dyadic cubes $Q$ we have
\begin{equation*}
\Mm^s_{\mu,\infty}(F\cap Q)=\Mm^s_{\mu,\infty}(Q).
\end{equation*}

\item For all open sets $U$ we have 
\begin{equation*}
\Mm^s_{\mu,\infty}(F\cap U)=\Mm^s_{\mu,\infty}(U).
\end{equation*}

\item There exists a constant $0<c\leqslant1$ such that for all open sets $U$ we have
\begin{equation*}
\Mm^s_{\mu,\infty}(F\cap U)\geqslant c\Mm^s_{\mu,\infty}(U).
\end{equation*}

\item There exists a constant $0<c\leqslant1$ such that for all dyadic cubes $Q$ we have
\begin{equation*}
\Mm^s_{\mu,\infty}(F\cap Q)\geqslant c\Mm^s_{\mu,\infty}(Q).
\end{equation*}
\end{enumerate}
\medskip
Moreover, the same is true for $\Gg^s_\mu$-sets with the corresponding adaptations.
\end{theorem}

These classes of sets also share some of the properties with their euclidean counterpart \cite[Theorem C]{F3}.

\begin{theorem}\label{teoC}
Let $(X,\mu)$ be a metric space of homogeneous type with $d=\dim_\Hh X$ and take $0<s\leqslant d$.
\begin{enumerate}[(i)]
\item If $0<t\leqslant s$, then $\Jj^s_\mu\subset\Jj^t_\mu$.

\item If $F\subset E\subset X$ and $F\in\Jj^s_\mu$, then $E\in\Jj^s_\mu$.

\item If $X$ is complete, then $\Jj^s_\mu$ is closed by countable intersections.
\end{enumerate}
\medskip
The same properties hold for $\Gg^s_\mu$-sets.
\end{theorem}

If the measure $\mu$ has more regularity, in particular if it is Ahlfors regular (i.e. if $\mu(B(x,r))\sim r^d$ for all ball $B(x,r)$ in $X$) then a couple of properties can be added.

\begin{theorem}\label{teoCAl}
Let $(X,\mu)$ be an Ahlfors regular space of dimension $d$ and take $0<s\leqslant d$.
\begin{enumerate}[(i)]
\item If $F\in\Jj^s_\mu$, then $\Hh^s_\infty(F)>0$ and in particular $\dim_\Hh F\geqslant s$.

\item Let $(Y,\nu)$ be another Ahlfors regular space and $f:X\to Y$ a bi-Lipschitz homeomorphism. Then $F\in\Jj^s_\mu(X)$ if and only if $f(F)\in\Jj^s_\nu(Y)$. 

\end{enumerate}
\medskip
The same properties hold for $\Gg^s_\mu$-sets.
\end{theorem}

Point $(i)$ of Theorem \ref{teoC} in particular means that $\Jj^s_\mu(X)\subset\Gg^s_\mu(X)$. Additionally, the equivalence between $\Mm_\infty$ and $\Mm_{\Ll,\infty}$ in $\R^d$ implies $\Gg^s_\Ll(\R^d)=\Gg^s(\R^d)$, where as before $\Ll$ denotes the Lebesgue measure in $\R^d$. Altogether
\begin{equation*}
\Jj^s_\Ll(\R^d)\subset\Gg^s_\Ll(\R^d)=\Gg^s(\R^d).
\end{equation*}
We will see, again in Example \ref{ejemploreal}, that the first inclusion is strict. Therefore, although $\Jj^s_\Ll(\R^d)$ is composed of sets with Hausdorff dimension at least $s$ (by $(i)$ of Theorem \ref{teoCAl}), and it is closed under countable intersections (by $(iii)$ in Theorem \ref{teoC}) and by bi-Lipschitz maps (by $(ii)$ of Theorem \ref{teoCAl}), it is not maximal among classes of sets with those conditions since $\Gg^s(\R^d)$ is strictly larger.

Classic examples of $\Gg^s$-sets in $\R$ are those obtained by Diophantine approximations. That is, for a fixed $\alpha>2$ we take the set $F_{\alpha}\subset\R$ of all real numbers $x$ such that
\begin{equation*}
\abs{x-\frac{p}{q}}\leqslant q^{-\alpha}
\end{equation*}
is verified for infinitely many rational numbers $p/q$. This set is in $\Gg^{1/\alpha}$, see e.g. \cite[Example 8.9]{F2}. In fact, even for a smaller set of rationals as the dyadic centers $p/q=k2^{-j}$ we still obtain a $\Gg^{1/\alpha}$-set. Points in this set have been called $\alpha$-approximable (by dyadics) and they are useful in the study of multifractal analysis of functions in Sobolev and Besov spaces \cite{AMS,FJ,JAF}.  
There are several generalizations and applications of these classical Diophantine approximations, we refer to \cite{BS,BDV,BF,B,DRV89,DRV} and the references therein.

Observe that the points $x\in F_\alpha$ can be expressed as the limit of a rational sequences $\{p_j/q_j\}$ where $0<q_j<q_{j+1}$ and $\abs{x-p_j/q_j}\leqslant(q_j)^{-\alpha}$. 
This can be naturally generalized to any metric space by what we call $(\Ee,\alpha)${\it-approximable points} (or simply $\alpha${\it-approximable points}): limits of sequences $\{x_j\}$ that converge at speed $\eps_j^\alpha$, where $\Ee=\{\eps_j\}$ is a sequence converging to $0$ and $\alpha$ is a real number bigger than $1$. For the precise definition see $\S$\ref{epsnet} below. 

As an application of the previous results we show that $\alpha$-approximable points are particular cases of $\Gg^s_\mu$-sets in certain regular measurable metric spaces.

\begin{theorem}\label{teoD}
Let $(X,\mu)$ a complete Ahlfors regular space of dimension $d$. Let $\alpha>1$ and $F\subset X$ a set of $(\Ee,\alpha)$-approximable points. Then $F\in\Gg^{d/\alpha}_\mu$ and therefore $\dim_\Hh F\geqslant d/\alpha$. Moreover, if the sequence $\Ee$ has exponential decay, then $\dim_\Hh F=d/\alpha$.
\end{theorem}
Let us remark that the lower bound for the Hausdorff dimension of these type of sets has been previously obtained on compact Ahlfors regular spaces in \cite[Theorem 2]{BDV} by using the concept of {\it ubiquitous systems}, first developed in \cite{DRV}. 

The upper bound in Theorem \ref{teoD}, however, is not true in general in the sense that exponential decay is required. For example, Jarník's Theorem \cite{J} shows that the Hausdorff dimension of the set of points obtained by Diophantine approximation on $\R$ for $\alpha>2$ (where the decay of convergence is polynomial) is actually $2/\alpha$ (in fact, it is also possible to prove that it is a $\Gg^{2/\alpha}$-set, see \cite[$\S10$]{F2}). In $\S$\ref{SecAproximables} we give an upper bound for the dimension in case of polynomial decay.

To end with, let us briefly describe how this work is organized. In Section 2 we give the basic definitions and properties on spaces of homogeneous type, along some preliminary results. In Sections 3 we prove Theorem \ref{teoB} and in Section 4 we prove Theorems \ref{teoC} and \ref{teoCAl}. In Section 5 we apply previous results to show that $\alpha$-approximable sets are a special case of $\Gg^s_\mu$-sets by proving Theorem \ref{teoD}. Finally in Section 6 we compare $\Jj^s_\mu$-sets with $\Gg^s_\mu$-sets and their links with previous established definitions \cite{B,F3}.

\section{Preliminaries}
\subsection{Doubling measures and doubling spaces}\label{doubling}
Let $X$ be a metric space. We denote the distance between two points $x,y\in X$ by $\abs{x-y}$, and the ball of center $x\in X$ and radius $r>0$ by $B(x,r):=\{y:\abs{y-x}\leqslant r\}$ (our balls will be always closed).

A Borel measure $\mu$ on $X$ is said to be {\it doubling} if it is finite and positive in every ball and there exists a constant $\C\geqslant1$ such that for every $x\in X$ and $r>0$,
\begin{equation*}
\mu(B(x,2r))\leqslant\C\mu(B(x,r)).
\end{equation*}
This inequality can be reformulated as
\begin{equation}\label{homogeneo}
\mu(B(x,R))\leqslant\C\left(\frac{R}{r}\right)^\beta\mu(B(x,r))
\end{equation}
for any $0<r<R$ and $x\in X$, and where $\beta\geqslant1$ is some independent constant (see e.g. \cite[Lemma 4.7]{H}).  
In this case we also say that $(X,\mu)$ is a metric space of {\it homogeneous type}.

A metric space $X$ is {\it (geometrically) doubling} if there exists a constant $N\in\N$ such that every bounded subset $A\subset X$ can be covered by no more than $N$ subsets with diameter at most $\diam(A)/2$. A metric space of homogeneous type is always doubling. Reciprocally, every complete doubling space admits a doubling measure, see \cite[$\S13$]{H}.



We can mimic the proof of Bolzano-Weierstrass property on $\R^d$ to get the following result:

\begin{lemma}\label{seqcomp} Let $X$ be a complete doubling metric space. Then every bounded and closed subset of $X$ is compact.
\end{lemma}



\subsection{Dyadic decomposition}\label{DescDiadica}

By a {\it dyadic decomposition} of parameter $0<\delta<1$ on a metric space $X$ we mean a family $\Qq$ of measurable subsets $Q_{j,k}$, called {\it dyadic cubes}, with $j\in\Z$ and $k\in I_j$ satisfying the following properties:
\begin{itemize}[label=---]
\item For every $j\in\Z$
\begin{equation}\label{partition}
X=\bigcup_{k\in I_j}Q_{j,k}.
\end{equation}

\item There is a uniform constant $\K\geqslant 1$ such that every dyadic cube $Q_{j,k}$ is a $\K$-quasi-ball with radius $\delta^j$. That is, there exists a point $x\in Q_{j,k}$ such that
\begin{equation}\label{balls}
B(x,\delta^j)\subset Q_{j,k}\subset B(x,\K\delta^j).
\end{equation}

\item If $i\geqslant j$, then
\begin{equation}\label{disj}
\text{either }\:Q_{i,k'}\subset Q_{j,k}\:\text{ or }\:Q_{i,k'}\cap Q_{j,k}=\emptyset.
\end{equation}
\end{itemize}
We will say that the {\it level} of an arbitrary dyadic cube $Q$ is $j\in\N$, and denote $\lev(Q)=j$, if there exists $k\in I_j$ such that $Q=Q_{j,k}$.

As said in the introduction, due to \cite{HK}, for an appropriate choice of $\delta$ and $\K$ one can always construct a dyadic decomposition on a metric space of homogeneous type. Notice that, if $X$ is bounded, there exitst $j_0\in\Z$ such that $\#I_j=1$ and for every $j\leqslant j_0$ the dyadic cube of level $j$ is the whole space $X$. 

\subsection{Hausdorff and net outer measures}\label{haus}
In a general metric space $X$ the {\it Hausdorff outer measure} for any $s\geqslant0$ is defined in the following way: given $E\subset X$,
\begin{equation*}
\Hh^s(E):=\lim_{r\downarrow0}\Hh^s_r(E),\quad\text{with }\Hh^s_r(E):=\inf\left\{\sum_n\diam(E_n)^s:E\subset\bigcup_n E_n,\diam(E_n)\leqslant r\right\}.
\end{equation*}
The {\it Hausdorff content} of a set $E$ is $\Hh^s_r(E)$ for $r=\infty$. From this definition it follows that if $\Hh^s(E)<\infty$, then $\Hh^t(E)=0$ for every $t>s$. The {\it Hausdorff dimension} of a subset $E\subset X$ is then defined as
\begin{equation}\label{DefDimH}
\dim_\Hh E:=\inf\{s\geqslant0:\Hh^s(E)=0\}=\sup\{s\geqslant0:\Hh^s(E)=\infty\}.
\end{equation}
Using Carathéodory's Theorem for metric spaces one can see that $\Hh^s$ defines a Borel measure on $X$ for every $s\geqslant0$.


Next, in a metric space of homogeneous type $X$ with a dyadic decomposition $\Qq$ and Hausdorff dimension $d$ (which is always finite, see e.g. \cite[$\S$10]{H}) we define a {\it net outer measure} for any $s\geqslant0$ as $\Mm^s_\mu(E):=\lim_{r\downarrow0}\Mm^s_{\mu,r}(E)$, where
\begin{equation*}
\Mm^s_{\mu,r}(E):=\inf\left\{\sum_n\mu(Q_n)^{s/d}:E\subset\bigcup_n Q_n,\ Q_n\text{ is a dyadic cube with }\mu(Q_n)\leqslant r\right\}.
\end{equation*}
We also define the \textit{net content} of a subset $E\subset X$ as $\Mm_{\mu,r}^s(E)$ with $r=\infty$.



Different dyadic decomposition may give different net measures and contents but they are nonetheless equivalent:

\begin{proposition}\label{equivD}
Let $s\geqslant0$, and $\Qq$ and $\widetilde{\Qq}$ be two dyadic decompositions on a metric space of homogeneous type $(X,\mu)$ and $\Mm^s_\mu$ and $\widetilde{\Mm}^s_\mu$ their respective net measures. Then, for every $s\geqslant0$ there exists a constant $L\geqslant1$, such that
\begin{equation*}
L^{-1}\widetilde{\Mm}^s_\mu\leqslant\Mm^s_\mu\leqslant L\widetilde{\Mm}^s_\mu.    
\end{equation*}
The same is true for the (respective) contents. 
\end{proposition}

Before proving Proposition \ref{equivD} let us see an elementary lemma that will be also useful in the future.  

\begin{lemma}\label{lemaCC}
Fix a dyadic decomposition $\{Q_{j,k}\}$ of parameter $\delta$ on a metric space of homogeneous type $(X,\mu)$. There exists an increasing function $H:(0,+\infty)\to (1,+\infty)$ such that the number of dyadic cubes of level $j$ that intersect a ball of radius $R$ is at most $H(R/\delta^j)$. Furthermore, if $Q$ and $Q'$ are two of these cubes, then
\begin{equation}\label{RelCubos}
H(R/\delta^j)^{-1}\mu(Q)\leqslant\mu(Q') \leqslant H(R/\delta^j)\mu(Q).    
\end{equation}
\end{lemma}

\begin{proof}
Let $\K$ be the constant as in \eqref{balls} and denote by $\{Q_i\}_{i\in I}$ the set of dyadic cubes of level $j$ that intersect a ball $B(x,R)$. Then the triangle inequality gives $Q_i\subset B':=B(x,R+2\K\delta^j)$ for every $i\in I$.

Now we consider a set of points $\{x_i\}_{i\in I}$ such that $B(x_i,\delta^j)\subset Q_i$. Then using \eqref{homogeneo} we have
\begin{align*}
\mu(B')\leqslant\mu(B(x_i,2(R+2\K\delta^j)))&\leqslant\C2^\beta\left(\frac{R+2\K\delta^j}{\delta^j}\right)^\beta \mu(B(x_i,\delta^j))
\\
&\leqslant\C2^\beta\left(\frac{R+2\K\delta^j}{\delta^j}\right)^\beta \mu(Q_i).
\end{align*}
Defining $H(t):=\C2^\beta(t+2\K)^\beta$ we immediately get \eqref{RelCubos} since
\begin{equation*}
\mu(Q_i)\leqslant\mu(B')\leqslant H(R/\delta^j)\mu(Q_i).
\end{equation*}
Together with the fact that the cubes $\{Q_i\}_{i\in I}$ are disjoint this also yields
\begin{equation*}
\frac{\#I}{H(R/\delta^j)}\mu(B')\leqslant\sum_{i\in I}\mu(Q_i)\leqslant\mu(B'),    
\end{equation*}
and as a consequence $\#I\leqslant H(R/\delta^j)$.
\end{proof}

Observe that in Lemma \ref{lemaCC} one can change the ball of radius $R$ for an arbitrary set of diameter $R$ and $H(R/\delta^j)$ still works as an upper bound. 

\begin{proof}[Proof of Proposition \ref{equivD}]
Let $\delta_1$ and $\delta_2$ the parameters of $\Qq$ and $\widetilde{\Qq}$, and $\K_1$ and $\K_2$ their quasi-ball constants as in \eqref{balls}. Fix a subset $E\subset X$ and $r\in(0,+\infty]$. We take $\{Q_n\}$ a covering of $E$ by dyadic cubes of $\Qq_1$ whose measures do not exceed $r$ and denote $j_n=\lev(Q_n)$.

For every $n$ we take $\ell_n\in\Z$ such that $\delta_2^{\ell_n}\leqslant\delta_1^{j_n}< \delta_2^{\ell_n-1}$. Using Lemma \ref{lemaCC} and the fact that $Q_n$ is included in a ball of radius $\K_1\delta_1^{j_n}$ we can cover $Q_n$ with no more than $H_2(\K_1\delta_1^{j_n}/\delta_2^{\ell_n})\leqslant H_2(\K_1/\delta_2)$ cubes of level $\ell_n$ in $\widetilde{\Qq}$, where $H_2$ is the increasing function given by Lemma \ref{lemaCC} for the covering $\widetilde{\Qq}$. We denote these cubes by $\{P_{n,i}\}$. 

If $B(x_n,\delta_1^{j_n})\subset Q_n\subset B(x_n,\K_1\delta^{j_n})$ are the associated balls, then for every $i$ we get $P_{n,i}\subset B(x_n,\K_1 \delta_1^{j_n}+ 2\K_2 \delta_2^{\ell_n})$. Using \eqref{homogeneo} we have
\begin{align*}
\mu(P_{n,i})&\leqslant\mu(B(x_n,\K_1\delta_1^{j_n}+2\K_2\delta_2^{\ell_n}))
\\
&\leqslant\C\left(\frac{\K_1\delta_1^{j_n}+2\K_2\delta_2^{\ell_n}}{\delta_1^{j_n}}\right)^\beta\mu(B(x_n,\delta_1^{j_n}))
\\
&\leqslant\C\left(\K_1+2\K_2\right)^\beta\mu(Q_n).    
\end{align*}
This implies
\begin{equation*}
\sum_n\sum_i\mu(P_{n,i})^{s/d}\leqslant\sum_nH_2(\K_1/\delta_2)\C^{s/d}\left(\K_1+2\K_2\right)^{\beta s/d} \mu(Q_n)^{s/d}.
\end{equation*}
Since $\mu(Q_n)\leqslant r$, by definition we have $\sum_{n,i}\mu(P_{n,i})^{s/d}\geqslant\widetilde{\Mm}^s_{\mu,r'}(E)$, where $r'=\C\left(\K_1+2\K_2\right)^\beta r$. Thus, taking the infimum over all $\Qq_1$-dyadic decompositions we obtain
\begin{equation*}
\widetilde{\Mm}^s_{\mu,r'}(E)\leqslant H_2(\K_1/\delta_2)\C^{s/d}\left(\K_1+2\K_2\right)^{\beta s/d}\Mm^s_{\mu,r}(E).
\end{equation*}
If $r=+\infty$ this means that
\begin{equation*}
\widetilde{\Mm}^s_{\mu,\infty}(E)\leqslant H_2(\K_1/\delta_2)\C^{s/d}\left(\K_1+2\K_2\right)^{\beta s/d}\Mm^s_{\mu,\infty}(E).    
\end{equation*}
If not, we take limits when $r\to0$ to obtain
\begin{equation*}
\widetilde{\Mm}^s_\mu(E)\leqslant H_2(\K_1/\delta_2)\C^{s/d}\left(\K_1+2\K_2\right)^{\beta s/d}\Mm^s_\mu(E).    
\end{equation*}
The reverse inequality follows symmetrically. Thus, if we define
\begin{equation*}
L:=\max\left\{H_2(\K_1/\delta_2)\C^{s/d}\left(\K_1+2\K_2\right)^{\beta s/d},H_1(\K_2/\delta_1)\C^{s/d}\left(\K_2+2\K_1\right)^{\beta s/d}\right\}.    
\end{equation*}
the desired inequalities are verified. 
\end{proof}

We will be mostly working with the content $\Mm^s_{\mu,\infty}$. Note that, given $0<s\leqslant d$,
\begin{equation}\label{dyadicmeasure}
\Mm^s_{\mu,\infty}(Q)=\mu(Q)^{s/d}
\end{equation}
holds for all dyadic cubes $Q$. This follows from the $\sigma$-additivity of $\mu$ and the sub-additivity of the function $x\mapsto x^\tau$ when $\tau\in(0,1]$.

\subsection{Ahlfors regular spaces}\label{Ahlforsspaces}
An Ahlfors regular space is a metric space $X$ with a Borel measure $\mu$ for which there exist two constants $d>0$ and $\A>0$ such that
\begin{equation}\label{ahl}
\A^{-1}r^d\leqslant\mu(B(x,r))\leqslant\A r^d
\end{equation}
holds for all $0<r\leqslant\diam(X)$ and $x\in X$. The constant $d$ is sometimes referred as the {\it dimension} of $X$ and indeed it coincides with the Hausdorff dimension as defined in $\S$\ref{haus}. Moreover, it can be proved that $C^{-1}\Hh^d\leqslant\mu\leqslant C\Hh^d$ for some constant $C>0$, see e.g. \cite[$\S8.7$]{H}. In fact, this translates to the equivalence between any Hausdorff measure and its net measure counterpart:

%



\begin{lemma}\label{EquivMH}
Let $(X,\mu)$ be an Ahlfors regular metric space of dimension $d$, then given $0<s\leqslant d$ there exists a constant $\C_0>0$ such that
\begin{equation*}
\C_0^{-1}\Hh^s(E)\leqslant\Mm^s_\mu(E)\leqslant\C_0\Hh^s(E),
\end{equation*}
holds for all $E\subset X$. The same inequality holds for the contents $\Mm^s_{\mu,\infty}$ and $\Hh^s_\infty$.
\end{lemma}

\begin{proof}
Let $\Qq$ be a dyadic decomposition with parameter $\delta\in (0,1)$ and quasi-ball constant $\K$. By slightly modifying $\A$ in \eqref{ahl}, we have that if $Q\in\Qq$ has level $\ell$ with $\delta^{\ell+1}\leqslant\diam(X)$, then
\begin{gather}
\label{diamQ}\diam(Q)\leqslant2\K\delta^\ell
\\
\label{volQ}\A^{-1}\delta^{\ell d}\leqslant\mu(Q)\leqslant\A\K^d\delta^{\ell d}.    
\end{gather}

Now take an arbitrary subset $E\subset X$ together with a dyadic covering $\{Q_n\}$ with $\mu(Q_n)\leqslant r$ for all $n$. We can assume that $\delta^{\lev(Q_n)+1}\leqslant \diam(X)$ for every $n$. Using \eqref{diamQ} and the lower bound of \eqref{volQ} we see that 
\begin{equation*}
\diam(Q_n)\leqslant2\K\A^{1/d}\mu(Q_n)^{1/d}\leqslant2\K\A^{1/d}r^{1/d}
\end{equation*}
holds for all $n$. Thus, if we set $r':=2\K\A^{1/d}r^{1/d}$ we have that
\begin{equation*}
\Hh^s_{r'}(E)\leqslant\sum_n\mathrm{diam}(Q_n)^s\leqslant2^s\K^s\A^{s/d}\sum_n\mu(Q_n)^{s/d}.
\end{equation*}
Taking infimum among all dyadic coverings $\{Q_n\}$ and then letting $r\to0$ we obtain $\Hh^s(E)\leqslant2^s\K^s\A^{s/d}\Mm^s_{\mu}(E)$. If we fix $r=\infty$, then $\Hh^s_{\infty}(E)\leqslant2^s\K^s\A^{s/d}\Mm^s_{\mu,\infty}(E)$.

Next, consider a covering $\{U_n\}$ of $E$ by sets with diameter smaller than $r$. For each $n$ let $\ell_n\in\Z$ such that 
\begin{equation}\label{aprox}
\K\delta^{\ell_n}\leqslant \mathrm{diam}(U_n)\leqslant\K\delta^{\ell_n-1},
\end{equation}
and cover $U_n$ with dyadic cubes of $\Qq$ with level $\ell_n$, denoted by $Q_{n,1},\dots,Q_{n,m_n}$. Observe that, by Lemma \ref{lemaCC}, $m_n\leqslant H(\K/\delta)$ for every $n$. Using \eqref{aprox} and the upper bound of \eqref{volQ} we have 
\begin{equation*}
\mu(Q_{n,i})\leqslant \A\diam(U_n)^d\leqslant\A r^d
\end{equation*}
for all $n$ and $i=1,\dots,m_n$. Denoting $r''=\A r^d$ we see that
\begin{equation*}
\Mm_{\mu,r''}^s(E)\leqslant\sum_n\sum_i\mu(Q_{n,i})^{s/d}\leqslant H(\K/\delta)\sum_n\diam(U_n)^s.
\end{equation*}
Taking infinum among the coverings $\{U_n\}$ and then letting $r\to 0$, we get $\Mm^s_\mu(E)\leqslant H(\K/\delta)\Hh^s(E)$. If we fix $r=\infty$, then $\Mm^s_{\mu,\infty}(E)\leqslant H(\K/\delta)\Hh^s_\infty(E)$.
\end{proof}
\subsection{$\alpha$-approximable points}\label{epsnet}
Let $X$ be a metric space and $\epsilon,c_1,C_1>0$. We say that a discrete subset $\{x_k\}\subset X$ is a $(c_1,C_1,\epsilon)$-net if
\begin{equation*}
\abs{x_k-x_{k'}}>2c_1\eps\quad k\neq k',\qquad\inf_k\abs{x_k-x}<C_1\eps\quad\forall x\in X.
\end{equation*}
The first condition says that the balls centered at points in $\{x_k\}$  with radius $c_1\epsilon$ are pairwise disjoint, and the second condition implies that the family of balls taken with the same centers but with radius $C_1\epsilon$ is a covering of $X$. Observe that if $C_1\geqslant 2c_1$, then one can use Zorn's lemma to construct $(c_1,C_1,\epsilon)$-nets on $X$.

Let $\alpha>1$ and $\Ee=\{\epsilon_j\}$ be a positive sequence with $\epsilon_j\to0$. For each $j$, let $\{x_{j,k}\}_k$ be a $(c_1,C_1,\eps_j)$-net, consider
\begin{equation*}
E_j:=\bigcup_kB(x_{j,k},\eps_j^\alpha),
\end{equation*}
and then define the set of $(\Ee,\alpha)$-approximable points as
\begin{equation*}
F:=\limsup E_j=\bigcap_j\bigcup_{j'\geqslant j}E_{j'}.
\end{equation*}

If $X$ is doubling then $F$ can also be defined as the set of points $y\in X$ such that
\begin{equation*}
\abs{x_{j,k}-y}\leqslant\epsilon_j^\alpha
\end{equation*}
holds for infinitely many points $x_{j,k}$, which is the same condition defining  Diophatine approximations.

\section{Characterization of $\Jj^s_\mu$-sets}
We begin by proving the equivalence provided by Theorem \ref{teoB}. Throughout this section we will assume that $(X,\mu)$ is a metric space of homogeneous type and that $\C>0,\beta>0$ are constants that satisfy \eqref{homogeneo}. Further, we will take a dyadic decomposition $\Qq$ with parameter $\delta\in(0,1)$ and quasi-ball constant $\K$ as in \eqref{balls}.

\begin{proof}[Proof of Theorem \ref{teoB}]
Observe that once the theorem is proved for $\Jj^s_\mu$-sets, it automatically follows for $\Gg^s_\mu$-sets. Let us then prove the equivalences for $\Jj^s_\mu$-sets. 

Note that $(ii)\to(iii)$ is immediate by taking $c=1$. The direction $(iii)\to(iv)$ follows from the fact that 
\begin{equation}\label{cubeint}
\Mm^s_{\mu,\infty}(\inte(Q))\geqslant c_0\Mm^s_{\mu,\infty}(Q)
\end{equation}
for all dyadic cubes $Q$, where $c_0>0$ is an independent constant. Indeed, if \eqref{cubeint} holds for all $Q$, then for every $F$ verifying $(iii)$ we can find a constant $c>0$ such that 
\begin{equation*}
\Mm^s_{\mu,\infty}(F\cap Q)\geqslant\Mm^s_{\mu,\infty}(F\cap\inte(Q))\geqslant c\Mm^s_{\mu,\infty}(\inte(Q))\geqslant c_0c\Mm^s_{\mu,\infty}(Q)
\end{equation*}
for all dyadic cubes $Q$. Now, to see \eqref{cubeint} note that if $j$ is the level of $Q$ and $\{Q_n\}$ is a family of dyadic cubes covering $\inte(Q)$, then
\begin{align*}
\sum_n\mu(Q_n)^{s/d}\geqslant \mu(\inte(Q))^{s/d}&\geqslant\mu(B(x_Q,\delta^j/2))^{s/d}
\\
&\geqslant\C^{-s/d}(2\K)^{-\beta s/d}\mu(B(x_Q,\K\delta^j))^{s/d}
\\
&\geqslant\C^{-s/d}(2\K)^{-\beta s/d}\mu(Q)^{s/d}
\\
&\geqslant\C^{-s/d}(2\K)^{-\beta s/d}\Mm^s_{\mu,\infty}(Q),
\end{align*}
where $x_Q\in Q$ verifies \eqref{balls} and where we have also used \eqref{homogeneo} in the third inequality and \eqref{dyadicmeasure} in the last one. Finally, \eqref{cubeint} follows by taking $c_0=\C^{-s}(2\K)^{-\beta s}$.

\medskip

Let us show $(i)\to(ii)$. Here we essentially reproduce \cite[Lemma 1]{F3} to our current setting. Assume that $F\subset X$ verifies $(i)$ and take an open subset $U\subset X$. We need to show that for every arbitrary dyadic covering $\{P_m\}$ of $F\cap U$ we have
\begin{equation*}
\sum_m\mu(P_m)^{s/d}\geqslant\Mm_{\mu,\infty}^s(U).
\end{equation*}
First, due to \eqref{disj}, we may assume that $\{P_m\}$ is disjoint.  Second, since $U$ is we can write $U=\bigcup_n Q_n$ where $\{Q_n\}$ is a family of disjoint dyadic cubes. Thus
\begin{equation*}
\bigcup_n F\cap Q_n\subset\bigcup_m P_m.
\end{equation*}
From here, and since $F\cap Q\neq\emptyset$ for any dyadic cube $Q$, we may assume, using the nested property \eqref{disj}, that for each $Q_n$ one (and only one) of the following is true:
\begin{enumerate}[(a)]
\item there exists a unique $m$ such that 
\begin{equation*}
Q_n\subset P_m,
\end{equation*}

\item $Q_n$ is not included in any $P_m$ and instead
\begin{equation*}
F\cap Q_n\subset\bigcup_{P_m\subset Q_n}P_m.
\end{equation*}
\end{enumerate}

If $Q_n$ verifies (a), then clearly
\begin{equation*}
\sum_{P_m\cap Q_n\neq\emptyset}\mu(P_m)^{s/d}\geqslant\mu(Q_n)^{s/d}.
\end{equation*}
If $Q_n$ verifies (b), then
\begin{equation*}
\sum_{P_m\cap Q_n\neq\emptyset}\mu(P_m)^{s/d}\geqslant\Mm^s_{\mu,\infty}(F\cap Q_n)=\Mm^s_{\mu,\infty}(Q_n)=\mu(Q_n)^{s/d},
\end{equation*}
since we are assuming that $F$ verifies $(i)$ and by also using \eqref{dyadicmeasure} in the last equality. Thus, changing all those $P_m$ that are proper subcubes of some $Q_n$ by one copy of $Q_n$, and keeping the rest of the cubes $P_m$, we obtain a dyadic covering $\{K_l\}$ of $U$ such that
\begin{equation*}
\sum_m\mu(P_m)^{s/d}\geqslant\sum_l\mu(K_l)^{s/d}\geqslant\Mm^s_{\mu,\infty}(U),
\end{equation*}
which gives us exactly what we wanted to prove.

\medskip

Finally, let us show that $(iv)\to(i)$. Here we adapt the ideas of \cite[Lemma 2]{FP}. Assume that $F$ verifies $(iv)$ and let $Q$ be an arbitrary dyadic cube. We want to show that for any arbitrary dyadic covering $\{Q_n\}$ of $F\cap Q$ we have
\begin{equation*}
\sum_n\mu(Q_n)^{s/d}\geqslant\Mm^s_{\mu,\infty}(Q).
\end{equation*}
As before, we can suppose that this covering is disjoint and further that $Q_n\subset Q$ for all $n$. Moreover, since $\Mm^s_{\mu,\infty}(F\cap Q)$ is finite (e.g. it can be covered by just $Q$) we may assume that 
\begin{equation*}
\sum_n\mu(Q_n)^{s/d}<\infty.
\end{equation*}
Next, as the series is convergent, we can reorganize the sum according to the level of each cube: for each $j\geqslant \lev(Q)$, denote $L_j:=\{n:\lev(Q_n)=j\}$ so that
\begin{equation*}
\sum_n\mu(Q_n)^{s/d}=\sum_{j}\sum_{n\in L_j}\mu(Q_n)^{s/d}.
\end{equation*}
Moreover, since $Q_n\subset Q$ for all $n$ and the subcubes are all disjoint, then an application of Lemma \ref{lemaCC} shows that $L_j$ is finite for any $j$. Altogether this means that, given $\eps>0$, we can find $j_0$ such that
\begin{equation}\label{colaserie}
\sum_{j\geqslant j_0}\sum_{n\in L_j}\mu(Q_n)^{s/d}
<\eps.
\end{equation}
We then construct the following dyadic partition of $Q$: let $\{P_m\}$ such that for each $m$ either
\begin{enumerate}
\item[(c)] $P_m=Q_n$ for some $n$ and $\lev(P_m)<j_0$,

\item[(d)] $\lev(P_m)={j_0}$ and all the $Q_n$ that intersect $F\cap P_m$ are contained in $P_m$.
\end{enumerate}
(This can be done from a $j_0$-level partition of dyadic subcubes of $Q$.)
Let us first take $m_0$ such that $P_{m_0}$ satisfies (c) so that
\begin{equation*}
\sum_{Q_n\subset P_{m_0}}\mu(Q_n)^{s/d}=\mu(P_{m_0})^{s/d}.
\end{equation*}
In particular this means that
\begin{equation}\label{caso1}
\sum_{j<j_0}\sum_{n\in L_j}\mu(Q_n)^{s/d}=\sum_{\lev(P_m)<j_0}\mu(P_m)^{s/d}.
\end{equation}
If now $P_{m_0}$ satisfies (d) we then have that
\begin{equation*}
\sum_{Q_n\subset P_{m_0}}\mu(Q_n)^{s/d}\geqslant\Mm^s_{\mu,\infty}(F\cap P_{m_0})\geqslant c\Mm^s_{\mu,\infty}(P_{m_0})=c\mu(P_{m_0})^{s/d}
\end{equation*}
by using the fact that $F$ verifies $(iv)$ in the last inequality. This gives
\begin{equation}\label{caso2}
\sum_{j\geqslant j_0}\sum_{n\in L_j}\mu(Q_n)^{s/d}\geqslant c\sum_{\lev(P_m)=j_0}\mu(P_m)^{s/d}.
\end{equation}
Altogether \eqref{colaserie}, \eqref{caso1} and \eqref{caso2} yield
\begin{align*}
\sum_j\sum_{n\in L_j}\mu(Q_n)^{s/d} &=\sum_{j<j_0}\sum_{n\in L_j}\mu(Q_n)^{s/d}+\sum_{j\geqslant j_0}\sum_{n\in L_j}\mu(Q_n)^{s/d}
\\
&=\sum_{\lev(P_m)<j_0}\mu(P_m)^{s/d}+(c^{-1}+1-c^{-1})\sum_{j\geqslant j_0}\sum_{n\in L_j}\mu(Q_n)^{s/d}
\\
&\geqslant\sum_{\lev(P_m)<j_0}\mu(P_m)^{s/d}+\sum_{\lev(P_m)=j_0}\mu(P_m)^{s/d}+(1-c^{-1})\eps
\\
&\geqslant\Mm^s_{\mu,\infty}(Q)+(1-c^{-1})\eps,
\end{align*}
by also using that $\{P_m\}$ is a dyadic partition of $Q$ with cubes of level at most $j_0$ in the last inequality. Finally, since $\eps>0$ is arbitrary we conclude that $\sum_n\mu(Q_n)^{s/d}\geqslant\Mm^s_{\mu,\infty}(Q)$.
\end{proof}

To end this section let us note that characterizations $(iii)$ or $(iv)$ of Theorem \ref{teoB} together with Proposition \ref{equivD} show that Definition \ref{jsets} is independent from the chosen dyadic partition. 

\section{Properties of $\Jj^s_\mu$-sets}

We begin by showing different properties of the classes $\Jj^s_\mu$ and $\Gg^s_\mu$ provided by Theorem \ref{teoC}. Here we take the same assumptions and use the same notation than in the previous section.

\begin{proof}[Proof of Theorem \ref{teoC}]
As in Theorem \ref{teoB} it is enough to show the results for the class $\Jj^s_\mu$.

$(i)$ If $t\leqslant s$, then for all families of dyadic cubes $\{Q_n\}$ we have
\begin{equation*}
\sum_n\mu(Q_n)^{t/d}\geqslant\left(\sum_n\mu(Q_n)^{s/d}\right)^{t/s},
\end{equation*}
which yields $\Mm^t_{\mu,\infty}(E)\geqslant\left(\Mm^s_{\mu,\infty}(E)\right)^{t/s}$ for all subset $E\subset X$. Together with \eqref{dyadicmeasure} this implies $\Jj^s_\mu\subset\Jj^t_\mu$.

\medskip

$(ii)$ This is obvious from the definition of $\Jj^s_\mu$-set and the fact that $\Mm^s_\mu$ is an outer measure.

\medskip

\medskip

$(iii)$ Here we adapt Lemma 4 of \cite{F3}. Let $\{F_j\}_{j\geqslant1}$ be a sequence of sets belonging to the class $\Jj_\mu^s$. We will prove that $\bigcap_j F_j\in \Jj^s_\mu$.

We may assume that $\{F_j\}_{j\geqslant1}$ is a decreasing sequence of open sets. To see this suppose that $\{F_j\}_{j\geqslant1}$ is arbitrary and use the $G_\delta$ property to write each $F_j$ as the countable intersection of open sets. By $(ii)$ of this Theorem we know that each of those open sets must belong to $\Jj^s_\mu$, and due to $(ii)$ of Theorem \ref{teoB} the finite intersection of open $\Jj^s_\mu$-sets belongs to $\Jj^s_\mu$. Then a diagonal argument allows us to write $\bigcap_{j\geqslant1}F_j$ as the decreasing intersection of open $\Jj^s_\mu$-sets.

Next, take an open subset $U\subset X$. Suppose first that $U$ is bounded and let $\eps>0$. We will construct a decreasing sequence of open sets $\{U_j\}_{j\geqslant1}$ such that for each $j\geqslant1$
\begin{gather}
\label{clau}\overline{U_j}\subset F_j\cap U
\\
\label{approxbyU}\Mm^s_{\mu,\infty}(U_j)>\Mm^s_{\mu,\infty}(U)-\eps.
\end{gather}
To that end, given a sequence of positive real numbers $\{r_j\}$ we set
\begin{equation*}
U_0=U,\qquad U_j=(F_j\cap U_{j-1})_{(-r_j)}\quad\text{for }j\geqslant1,
\end{equation*}
where for any given subset $S\subset X$ and $r>0$ we denote $S_{(-r)}:=\{x\in S:\inf_{y\notin S}\abs{x-y}>r\}$ (which is always open). Note first $U_j$ verifies \eqref{clau} for any $j\geqslant1$. Secondly, we claim that if the sequence $\{r_j\}$ is chosen appropriately, then \eqref{approxbyU} is also met. Indeed if $U_{j-1}$ verifies \eqref{approxbyU}, then $(ii)$ of Theorem \ref{teoB} applied to $F_j$ gives
\begin{equation*}
\Mm^s_{\mu,\infty}(F_j\cap U_{j-1})=\Mm^s_{\mu,\infty}(U_{j-1})>\Mm^s_{\mu,\infty}(U)-\eps.
\end{equation*}
Finally, as $(F_j\cap U_{j-1})_{(-r)}\nearrow F_j\cap U_{j-1}$ when $r\searrow0$, the Increasing Sets Lemma for net measures \cite[Theorem 52]{R} assures the existence of a small enough $r_j>0$ such that $U_j=(F_j\cap U_{j-1})_{(-r_j)}$ verifies \eqref{approxbyU}.

Using \eqref{clau} we see that if $\{Q_n\}$ is a dyadic covering of $\bigcap_{j\geqslant 1}F_j\cap U$ then
\begin{equation*}
\bigcap_{j\geqslant1}\overline{U}_j\subset\bigcup_nQ_n\subset\bigcup_nB(x_n,\K\delta^{j_n}),
\end{equation*}
where $j_n$ denotes the level of $Q_n$ and $B(x_n,\K\delta^{j_n})$ is the ball on the right hand side of \eqref{balls}. Since $X$ is complete and $\overline{U_j}$ is closed and bounded for all $j$, then Lemma \ref{seqcomp} implies that it must be compact. Further, since the sequence $\{\overline{U}_j\}$ is decreasing, there exits $j_1\geqslant1$ for which $\overline{U}_{j_1}\subset\bigcup_nB(x_n,\K\delta^{j_n})$. 
Finally, for each $n$ let $\{Q_{n,m}\}_m$ the family of cubes with level $j_n$ that intersect the ball $B(x_n,\K\delta^{j_n})$ (which clearly contains $Q_n$). Thus $\overline{U}_{j_1}\subset\bigcup_n\bigcup_mQ_{n,m}$. In particular this means that
\begin{equation*}
\sum_n\sum_m\mu(Q_{n,m})^{s/d}\geqslant\Mm^s_{\mu,\infty}(U_{j_1}).
\end{equation*}
Now, we apply Lemma \ref{lemaCC} by comparing each $Q_{n,m}$ to $Q_n$ to obtain
\begin{equation*}
\sum_n\sum_m\mu(Q_{n,m})^{s/d}\leqslant\sum_n H(\K\delta^{j_n}/\delta^{j_n})^{s/d}\mu(Q_n)^{s/d}\leqslant H(\K)^{s/d}\sum_n\mu(Q_n)^{s/d}
\end{equation*}
Altogether, taking $c=H(\K)^{-s}$ and using \eqref{approxbyU}, the previous inequalities read
\begin{equation*}
\sum_n\mu(Q_n)^{s/d}\geqslant c\Mm^s_{\mu,\infty}(U_j)>c(\Mm^s_{\mu,\infty}(U)-\eps).
\end{equation*}
Finally, since $\{Q_n\}$ and $\eps>0$ are arbitrary we get
\begin{equation*}
\Mm^s_{\mu,\infty}\left(\bigcap_{j\geqslant1}F_j\cap U\right)\geqslant c\Mm^s_{\mu,\infty}(U),
\end{equation*}

If $U$ is unbounded one can use again the Increasing Sets Lemma to obtain the same inequality. Therefore, in any case, $(iii)$ of Theorem \ref{teoB} is verified for $\bigcap_{j\geqslant1}F_j$ and so $\bigcap_{j\geqslant1}F_j\in\Jj^s_\mu$.
\end{proof}

We now proceed to prove Theorem \ref{teoCAl}. Here we are assuming in addition that $(X,\mu)$ is Ahlfors regular with dimension $d$.

\begin{proof}[Proof of Theorem \ref{teoCAl}]
$(i)$ Using Lemma \ref{EquivMH} we have that if $F\in \Jj_\mu^s$, then
\begin{align*}
\Hh^s_\infty(F)\geqslant\C_0^{-1}\Mm^s_{\mu,\infty}(F)
&\geqslant\C_0^{-1}\Mm^s_{\mu,\infty}(F\cap Q)
\\
&\geqslant\C_0^{-1}\Mm^s_{\mu,\infty}(Q)=\C_0^{-1}\mu(Q)^{s/d}>0
\end{align*}
by using $(i)$ of Theorem \ref{teoB} in the third equality for an arbitrary dyadic cube $Q$.

\medskip

$(ii)$ Suppose that $F\in\Jj^s_\mu(X)$. By $(iii)$ of Theorem \ref{teoB} and Lemma \ref{EquivMH}, to prove that $f(F)\in\Jj^s_\nu(Y)$ it is enough to see that 
\begin{equation*}
\Hh^s_{Y,\infty}(f(F)\cap U)\geqslant c\Hh^s_{Y,\infty}(U)
\end{equation*}
holds for some constant $c>0$ and every open set $U\subset Y$.

Since $f$ is bi-Lipschitz, there exists $c_f>0$ such that 
\begin{equation*}
c_f^{-1}\Hh^s_{Y,\infty}(f(E))\leqslant\Hh^s_{X,\infty}(E)\leqslant c_f\Hh^s_{Y,\infty}(f(E))
\end{equation*}
holds for every subset $E\subset X$. Then, using again $(iii)$ of Theorem \ref{teoB} and Lemma \ref{EquivMH}, we have
\begin{align*}
\Hh_{Y,\infty}^s(f(F)\cap U)&\geqslant c_f^{-1}\Hh_{X,\infty}^s(F\cap f^{-1}(U))
\\
&\geqslant\C_1^{-1}c_f^{-1}\Mm^s_\mu(F\cap f^{-1}(U))
\\
&\geqslant\C_1^{-1}c_f^{-1}\Mm^s_\mu(f^{-1}(U))
\\
&\geqslant\C_1^{-2}c_f^{-1}\Hh_{X,\infty}^s(f^{-1}(U))
\\
&\geqslant\C_1^{-2}c_f^{-2}\Hh_{Y,\infty}^s(U)
\end{align*}
where $\C_1$ is the maximum of the $\C_0$'s of Lemma \ref{EquivMH} applied for $X$ and $Y$, and where $U\subset Y$ is an arbitrary open set. The other implication is analogous taking the function $f^{-1}$.
\end{proof}

Note that for the proof of $(i)$ we only use the right hand side of the inequality in Lemma \ref{EquivMH} and therefore the result applies to any space $X$ that verifies the right hand side of \eqref{ahl}. In fact, in \cite{P} this hypothesis is assumed to estimate dimensions of different types of sets.

\section{$\alpha$-approximable points}\label{SecAproximables}

Throughout this section we work with the same construction laid out in $\S\ref{epsnet}$: let $\alpha>1$, and for each $j$ let $\{x_{j,k}\}_k$ be a $(c_1,C_1,\eps_j)$-net, where $\eps_j\to0$, and set
\begin{equation*}
E_j:=\bigcup_kB(x_{j,k},\epsilon_j),\qquad F:=\limsup E_j.
\end{equation*}

The following property of approxiamble points is key in the first part of Theorem \ref{teoD}. It uses the strategies of \cite[Example 8.9]{F2}.

\begin{lemma}\label{LemaAprox}
For every $0<t<d/\alpha$ there exists a constant $c>0$ such that  
\begin{equation*}
\limsup_{j\to\infty}\Mm^t_{\mu,\infty}(Q\cap E_j)\geqslant c \Mm^t_{\mu,\infty}(Q),
\end{equation*}
for every dyadic cube $Q$.
\end{lemma}

In order to prove the previous lemma we will need the following result. 

\begin{lemma}\label{lemita}
Let $D$ be a $K$-quasi-ball of radius $R\leqslant \diam(X)$. Consider 
\begin{gather*}
\Ff_j(D):=\{k:B(x_{j,k},c_1\eps_j)\subset D\},
\\
\Ff^j(D):=\{k:B(x_{j,k},c_1\eps_j)\cap D\neq\emptyset\},
\end{gather*}
and let $n_j(D)=\#\Ff_j(D)$, $m_j(D)=\#\Ff^j(D)$.
Then there exist $j_0$ and a constant $C_2\geqslant 1$ such that
\begin{equation}\label{nkmk}
C_2^{-1}\eps_j^{-d}(R-(c_1+C_1)\eps_j)^d\leqslant n_j(D)\leqslant m_j(D)\leqslant C_2\eps_j^{-d}(KR+2c_1\eps_j)^d
\end{equation}
holds for all $j\geqslant j_0$.
\end{lemma}

\begin{proof}
Let $x_0\in X$ such that
\begin{equation}\label{CB}
B(x_0,R)\subset D\subset B(x_0,KR).
\end{equation}
To begin with, since $\epsilon_j\to 0$, we can take $j_0$ large enough so that $R-(c_1+C_1)\epsilon_j>0$ for all $j\geqslant j_0$. Remember that the family of balls $B(x_{j,k},C_1\eps_j)$ covers $X$ and observe that if $k\notin\Ff_j(D)$ then the triangle inequality and the left hand side of \eqref{CB} reads $B(x_{j,k},C_1\epsilon_j)\cap B(x_0,R-(c_1+C_1)\epsilon_j)=\emptyset$. Altogether this implies
\begin{equation*}
B(x_0,R-(c_1+C_1)\eps_j)\subset\bigcup_{k\in\Ff_j(D)}B(x_{j,k},C_1\eps_j).
\end{equation*}
Then, using \eqref{ahl} we have
\begin{equation*}
\A^{-1}(R-(c_1+C_1)\eps_j)^d\leqslant\mu(B(x_0,R-(c_1+C_1)\eps_j))\leqslant\sum_{k\in\Ff_j(D)}\mu\left(B(x_{j,k},C_1\eps_j)\right)\leqslant n_j(D)\A C_1^d\eps_j^d,
\end{equation*}
from where
\begin{equation*}
\A^{-2}C_1^{-d}\eps_j^{-d}(R-(c_1+C_1)\eps_j)^d\leqslant n_j(D).
\end{equation*}

For the last inequality of \eqref{nkmk} note that the balls $B(x_{j,k},c_1\eps_j)$ with $k\in\Ff^j(D)$ are all included in $B(x_0,KR+2c_1\eps_j)$, cf. \eqref{CB}. And since $\{x_{j,k}\}_k$ is a $(c_1,C_1,\eps_j)$-net, then those balls are pairwise disjoint. This yields
\begin{equation*}
m_j(D)\A^{-1}c_1^d\eps_j^d \leqslant \sum_{k\in\Ff^j(D)}\mu(B(x_{j,k},c_1\eps_j))\leqslant\mu(B(x_0,KR+2c_1\eps_j))\leqslant\A(KR+2c_1\eps_j)^d.
\end{equation*}
This proof finishes by taking $C_2=\max\{\A^2c_1^{-d},\A^2C_1^{d}\}$.
\end{proof}

Now we prove Lemma \ref{LemaAprox}.

\begin{proof}[Proof of Lemma \ref{LemaAprox}]
Let $Q$ be an arbitrary dyadic cube of level $\ell$, $R=\min\{\delta^{\ell},\diam(X)\}$ and $x_0\in Q$ such that  
\begin{equation}\label{D}
B(x_0,R)\subset Q\subset B(x_0,\K R).
\end{equation}

From now on we fix $j_0\in\N$ so that \eqref{nkmk} is verified for  $j\geqslant j_0$ when we put $D=Q$ and $K=\K$. Furthermore, since $\eps_j\to0$ when $j\to\infty$ and $\alpha>1$, we may also assume that $\eps_j^\alpha\leqslant\frac{1}{2}c_1\eps_j$ and $(c_1+C_1)\eps_j\leqslant\frac{1}{2}R$ for all $j\geqslant j_0$. For each $j\geqslant j_0$ we consider $\nu_j$ the probability measure that distributes the mass among the balls of the family
\begin{equation*}
\mathcal{S}_j(Q)=\{B(x_{j,k},\eps_j^\alpha):k\in\Ff_j(Q)\}.
\end{equation*}
That is
\begin{equation*}
\nu_j(A\cap B)=\frac{\mu(A\cap B)}{\mu(B)n_j(Q)}
\end{equation*}
for every $B\in\mathcal{S}_j(Q)$.

\medskip

{\it Claim}: Let  $\varepsilon>0$. For every $j\geqslant j_0$ there exists $\{Q_n\}$ a dyadic covering of $Q\cap E_j$ that satisfies
\begin{equation}\label{claim}
\sum_n \mu(Q_n)^{t/d}\leqslant H(2\K/\delta)\left(\Mm_{\mu,\infty}^t(Q\cap E_j)+\varepsilon\right)    
\end{equation}
and for each $\ell_n=\lev(Q_n)$ either
\begin{enumerate}
\item[(a)] $2\K\delta^{\ell_n}\geqslant c_1\epsilon_j$, or
\item[(b)] $2\K\delta^{\ell_n}\leqslant \epsilon_j^\alpha$.
\end{enumerate}
In order to prove the claim let $\{P_m\}$ be a dyadic covering of $Q\cap E_j$ such that 
\begin{equation*}
\sum_m\mu(P_m)^{t/d}\leqslant \Mm_{\mu,\infty}^t(Q\cap E_j)+\varepsilon.
\end{equation*}
Next, let $p,q\in\N$ such that 
\begin{gather*}
2\K\delta^p\leqslant\epsilon_j^\alpha<2\K\delta^{p-1},
\\
2\K\delta^q< c_1\epsilon_j\leqslant 2\K\delta^{q-1}.
\end{gather*}
Suppose $P_m$ is such that $h_m=\lev(P_m)$ does not satisfies neither $(a)$ nor $(b)$ of the claim. Then we must have $q\leqslant h_m<p$ and so $\diam(P_m)\leqslant2\K\delta^{h_m}<c_1\eps_j$. Since $\epsilon_j^\alpha\leqslant \frac{1}{2}c_1\epsilon_j$ and the points of the form $x_{k,j}$ are $2c_1\epsilon_j$ separated, we have that $P_m$ intersects only one ball of $E_j$, which we denote by $B$. Let $P_{m,1},\ldots,P_{m,r_m}$ be the subcubes of $P_m$ of level $p$ that intersect $B$. By Lemma \ref{lemaCC} we have that
\begin{equation}\label{Estimacionrm}
r_m\leqslant H(\epsilon_j^\alpha/\delta^p)\leqslant H(2\K/\delta). 
\end{equation}
We define the covering $\{Q_n\}$ by changing all cubes $P_m$ as above by their respective subcubes $P_{m,1},\ldots,P_{m,r_m}$. Using that $\mu(P_{m,i})\leqslant \mu(P_{m})$ and \eqref{Estimacionrm} we get \eqref{claim} and then the claim is proved. 

\medskip

Now consider $\{Q_n\}$ as in the claim. Observe also that we can suppose that $\delta^{\ell_n+1}\leqslant R$ for every $n$. We want to estimate $\nu_j(Q_n)$ for each $n$. For that we study separately the cases given by the claim. 

If $n$ satisfies $(a)$ we can use Lemma \ref{lemita} to estimate for every $j\geqslant j_0$,
\begin{align*}
\nu_j(Q_n)\leqslant\frac{m_j(Q_n)}{n_j(Q)}&\leqslant \frac{C_2\eps_j^{-d}(\K\delta^{\ell_n}+2c_1\eps_j)^d}{C_2^{-1}\eps_j^{-d}(R-(c_1+C_1)\eps_j)^d}\\
&\leqslant C_2^2 \frac{\delta^{\ell_nt}}{R^t}\left(\frac{R^{t/d}(\K\delta^{\ell_n}+2c_1\eps_j)}{\delta^{\ell_nt/d}(R-(c_1+C_1)\eps_j)}\right)^d
\\
&
\leqslant C_2^2\frac{\delta^{\ell_nt}}{R^t}\left(\frac{\K\delta^{\ell_n(1-t/d)}R^{t/d}}{\frac{1}{2}R}+\frac{2c_1\eps_jR^{t/d}\delta^{-\ell_nt/d}}{\frac{1}{2}R}\right)^d
\\
&\leqslant2^dC_2^2\frac{\delta^{\ell_nt}}{R^t}\left(R^{t/d-1}\K\delta^{\ell_n(1-t/d)}+2^{1+t/d}R^{t/d-1}\K^{t/d}(c_1\eps_j)^{1-t/d}\right)^d,
\end{align*}
by also using $(c_1+C_1)\eps_j\leqslant\frac{1}{2}R$ in the third line and $\delta^{-\ell_n}\leqslant 2\K(c_1\epsilon_j)^{-1}$ in the last inequality. Since we also have $\delta^{\ell_n+1}\leqslant R$, we can deduce that for all $j\geqslant j_0$,
\begin{equation*}
\nu_j(Q_n)\leqslant 2^dC^2_2\frac{\delta^{\ell_nt}}{R^t} \left(\K\delta^{t/d-1}+2^{1+t/d}R^{(t/d-1)}\K^{t/d}(c_1\eps_j)^{1-t/d}\right)^d. 
\end{equation*}

\medskip

If $n$ satisfies $(b)$, then $Q_n$ intersects at most one ball $B(x,\eps_j^\alpha)$ in  $E_j$. We can use the that $X$ is Ahlfors-regular and Lemma \ref{lemita} in order to obtain, for every $j\geqslant j_0$,
\begin{align*}
\nu_j(Q_n)&\leqslant \frac{\mu(Q_n)}{\mu(B(x,\eps_j^\alpha))n_j(Q)}\leqslant \frac{\A \K^d\delta^{\ell_nd}}{\A^{-1}\eps_j^{(\alpha-1)d}C_2^{-1}(R-(c_1+C_1)\eps_j)^d}
\\
&\leqslant\A^2C_2\K^t\frac{\delta^{\ell_nt}}{R^t}\left(\frac{\eps_j^{1-\alpha}(2\K)^{1-t/d}\delta^{\ell_n(1-t/d)}R^{t/d}}{\frac{1}{2}R}\right)^d\\
&\leqslant 2^{d}\A^2C_2\K^t\frac{\delta^{\ell_nt}}{R^{t}} \left(\eps_j^{1-\alpha t/d}R^{t/d-1}\right)^d
\end{align*}    
by using again $(c_1+C_1)\eps_j\leqslant\frac{1}{2}R$.

\medskip

Now note that in any case, since $\alpha t<d$, then for every $\eta>0$ there exists $j_1\geqslant j_0$ such that if $j\geqslant j_1$, 
\begin{equation*}
\nu_j(Q_n)\leqslant(C_3+\eta)^d\frac{\delta^{\ell_nt}}{R^t}
\end{equation*}
holds for all $n$, where $C_3$ is a positive constant. Running the sum over all $n$ this yields
\begin{equation*}
1=\nu_j(Q\cap E_j)\leqslant\sum_n\nu_j(Q_n)\leqslant(C_3+\eta)^dR^{-t}\sum_n \delta^{\ell_nt}\leqslant\A(C_3+\eta)^dR^{-t}\sum_n\mu(Q_n)^{t/d},
\end{equation*}
where in the last inequality we use that $\delta^{\ell_n+1}\leqslant \diam(X)$ and \eqref{ahl} eventually modifying the constant $\A$.
Therefore
\begin{equation*}
R^t\leqslant 2\A (C_3+\eta)^{d} H(2\K/\delta)\left(\Mm_{\mu,\infty}^t(Q\cap E_j)+\varepsilon\right).
\end{equation*}
Taking $\limsup$ when $j\to\infty$ we obtain 
\begin{equation*}
\Mm^t_{\mu,\infty}(Q)=\mu(Q)^{t/d}\leqslant\A\K^tR^t\leqslant \A^2\K^tC_3^dH(2\K/\delta)\left(\limsup_{j\to\infty}\Mm_{\mu,\infty}^t(Q\cap E_j)+\varepsilon\right)
\end{equation*}
Since $\varepsilon$ is arbitrarily small, the proof finishes by putting $c=\left(\A^2\K^tC_3^dH(2\K/\delta)\right)^{-1}$.
\end{proof}

Before proceeding to prove Theorem \ref{teoD} let us precise what we mean when we say that the sequence of positive real numbers $\Ee=\{\eps_j\}$ has {\it exponential decay}, this is:
\begin{equation*}
\limsup_{j\to\infty}\frac{\log(\eps_j)}{j}<0.
\end{equation*}

\begin{proof}[Proof of Theorem \ref{teoD}]
Fix $t<d/\alpha$. By Lemma \ref{LemaAprox} there exists a constant $c\in(0,1)$ such that
\begin{equation*}
\limsup_{j\to\infty}\Mm^t_{\mu,\infty}(Q\cap E_j)\geqslant c \Mm^t_{\mu,\infty}(Q)
\end{equation*}
holds for every dyadic cube $Q$. Therefore, for every $j$ and every dyadic cube $Q$ we have
\begin{equation*}
\Mm^t_{\mu,\infty}\left(\left(\bigcup_{j'\geqslant j}E_{j'}\right)\cup Q\right)\geqslant \sup_{j'\geqslant j}\Mm^t_{\mu,\infty}(E_{j'}\cap Q)\geqslant c\Mm^t_{\mu,\infty}(Q).
\end{equation*}
This means that $\bigcup_{j'\geqslant j}E_{j'}$ belongs to $\Jj^t_\mu$ for every $j$. Finally, $(iii)$ of Theorem \ref{teoC} allows us to conclude that $F=\bigcap_j\bigcup_{j'\geqslant j}E_{j'}\in\Jj^t_\mu$. Since $t<d/\alpha$ was arbitrary, we get that $F\in\Gg^{d/\alpha}_\mu$.

\medskip

Now suppose that $\Ee$ has exponential decay and let $\beta>0$ be such that
\begin{equation*}
\limsup_{j\to\infty} \frac{\log (\eps_j)}{j}=-\beta.
\end{equation*}
Thus there exists $j_0$ such that $\eps_j\leqslant e^{-\beta j/2}$ for every $j\geqslant j_0$.

Now take $t>d/\alpha$. Consider an arbitrary ball $B=B(x_0,R)\subset X$ and let $0<r<R$. Let $j_1\geqslant j_0$ be such that $\eps_j^\alpha\leqslant c_1\eps_j$ and $2\eps_j^\alpha\leqslant r$ for every $j\geqslant j_1$, and take the covering of $F\cap B$ that consists of all balls $B(x_{j,k},\eps_j^\alpha)$ with $j\geqslant j_1$ that intersect $B$. Then applying the right hand side of \eqref{nkmk} for $B$ and that $\eps_j\leqslant r\leqslant R$ we have
\begin{align}\label{cotasup}
\Hh_r^t(F\cap B)&\leqslant\sum_{j\geqslant j_1}\sum_{k\in\Ff^j(B)}2^t\eps_j^{\alpha t}
\\
\nonumber&\leqslant\sum_{j\geqslant j_1}m_j(B)2^t\eps_j^{\alpha t}\leqslant 2^tC_2(1+2c_1)^dR^d \sum_{j\geqslant j_1}\eps_j^{\alpha t-d}
\end{align}
Since $\eps_j\leqslant e^{-\beta j/2}$ and $t>d/\alpha$, the series $\sum\eps_j^{\alpha t-d}$ converges. Taking the limit when $j_1\to 0$ we see that $\Hh_r^t(F\cap B)=0$ and then $\Hh^t(F\cap B)=0$, which shows that $\Hh^t(F)=0$ since $B$ is any ball. Finally, as $t>d/\alpha$ is also arbitrary, we have that $\dim_\Hh F\leqslant d/\alpha$.
\end{proof}

The estimate \eqref{cotasup} shows that if $\sum\eps_j^{\alpha t-d}$ converges, then $\dim_\Hh F\leqslant t$. We can use this fact to get upper bounds for the dimension in the case of non-exponential decay. For example, if $\eps_j=j^{-t}$ with $t>0$, it is easy to see that $\dim_\Hh F\leqslant\frac{d t+1}{\alpha t}$.

\section{Comparison between $\Jj^s_\mu$, $\Gg^s_\mu$ and other previous definitions}\label{defs}

As we said in the introduction, although $\Gg^s(\R^d)=\Gg_\Ll^s(\R^d)$, for a general metric space $X$ characterizing $\Gg^s_\mu$-sets as in \eqref{gsets}:
\begin{equation}\label{bilipdef}
\dim_\Hh\bigcap_if_i^{-1}(F)\geqslant s
\end{equation}
where $f_i:V\to X$ are bi-Lipschitz maps defined on an open subset $V\subset X$, may not be the best fit since it is possible that $X$ does not have so many similarities or bi-Lipschitz maps to begin with and so the condition \eqref{bilipdef} might be weaker than that of Definition \ref{jsets}. Indeed, let us illustrate this in the following example, which lies in the Ahlfors regular context.

\begin{example}\label{ejemplopegado}
Consider $X=H\cup L\subset\R^3$, where
\begin{equation*}
H=\{x=(x_1,x_2,x_3)\in\R^3 : x_3=0\},\qquad L=\{x=(x_1,x_2,x_3)\in \R^3 : x_1=x_2=0\}.
\end{equation*}
We define the distance on $X$ by:
\begin{equation*}
\abs{x-y}=\left\{\begin{array}{cc}\norm{x-y}&\text{if }x,y\in H
\\
\norm{x-y}^{1/2}&\text{if }x,y\in L
\\
\norm{x}+\norm{y}^{1/2}&\text{if }x\in H, y\in L 
\end{array}\right.
\end{equation*}
where $\norm{\cdot}$ denotes the usual Euclidean norm on $\R^3$. The metric space $X$ is thus obtained by gluing the Ahlfors $2$-regular spaces $(H,\norm{\cdot})$ and
$(L,\norm{\cdot}^{1/2})$, then it is Ahlfors $2$-regular (see e.g. \cite[Theorem 1.29]{G}). We consider on $X$ the Hausdorff measure $\mu=\Hh^2$. By obvious topological obstructions there are not bi-Lipschitz maps that carry open sets from $H$ onto $L$, which allows us to take sets in $X$ satisfying condition \eqref{bilipdef} that are not dense in $X$.

Specifically let $F\subset L$ be a set of $\alpha$-approximable points with $\alpha=2/s$. Applying Theorem \ref{teoD} in $L$ we get that $F\in\Gg_\mu^s(L)$. Thus, by a combination of $(i)$ and $(ii)$ of Theorem \ref{teoCAl} and $(iii)$ of Theorem \ref{teoC}, $F$ also satisfies \eqref{bilipdef}. But clearly $F\notin\Gg_{\mu}^s(X)$. Notice that this can be also written in terms of the $\Jj^t_\mu$-condition.

This example also shows that the class $\Gg_\mu^s(X)$ is not maximal among those classes of sets with dimension at least $s$ that are closed by countable intersections and bi-Lipschitz homeomorphisms. Indeed, the class $\Gg_\mu^s(X)\cup\Gg_\mu^s(L)$ is strictly larger and satisfies the same conditions.

Moreover, observe that with the same construction we can find two subsets $F,E\subset X$ that verify \eqref{bilipdef} but with $\dim_\Hh(F\cap E)=0$, so the class defined by this condition is not closed under countable intersections.
\end{example}


The following example shows that the family $\Gg_\Ll^s(\R^d)$ is strictly larger than $\Jj_\Ll^s(\R^d)$. In order to simplify the notation, in what follows we write $\Gg^s=\Gg_\Ll^s(\R^d)$ and $\Jj^s=\Jj_\Ll^s(\R^d)$ .

\begin{example}\label{ejemploreal}
Consider $X=\R^d$ with the euclidean distance and Lebesgue measure $\Ll$. Let $\alpha>1$ and set $s:=d/\alpha$. Define a sequence $\alpha_j=(1+1/\sqrt{j})\alpha$. Further, for each $j$ let $\{x_{j,k}\}_k$ be a $(1,2,e^{-j})$-net and set
\begin{equation*}
E_j:=\bigcup_kB(x_{j,k},e^{-j\alpha_j}),\qquad F:=\limsup E_j.
\end{equation*}
On the one hand, if $\beta>\alpha$, then the set of $\beta$-approximable points constructed from $\{x_{j,k}\}$ is contained in $F$. Then, by Theorem \ref{teoD} for every $t=d/\beta<d/\alpha=s$ we have that $F\in\Gg^t$. Whence $F\in \Gg^s$.

On the other hand, for every ball $B=B(x_0,R)\subset\R^d$ we can compute $\Hh^s_\infty(F\cap B)$ as in the proof of Theorem \ref{teoD}. For a given $j_1$ we consider the covering of $F\cap B$ consisting of all balls of the form $B(x_{j,k},e^{-j\alpha_j})$ with $j\geqslant j_1$ that intersect $B$. Next, we take $j_1$ large enough so that $e^{-j\alpha_j}\leqslant e^{-j}<R$ for every $j\geqslant j_1$. Thus, using the notation of Lemma \ref{lemita}, we see that the balls $B(x_{j,k},e^{-j\alpha_j})$ with $k\in\Ff^j(B)=\{k:B(x_{j,k},e^{-j})\cap B\neq0\}$ and $j\geqslant j_1$ cover $F\cap B$. In sum
\begin{equation*}
\Hh^s_\infty(F\cap B)\leqslant\sum_{j\geqslant j_1}\sum_{k\in\Ff^j(B)}\left(2e^{-j\alpha_j}\right)^s\leqslant 2^s\sum_{j\geqslant j_1}m_j(B) e^{-j(d+d/\sqrt{j})},
\end{equation*}
where, as before, $m_j(B)=\#\Ff^j(B)$. Using the right hand side of \eqref{nkmk} and the fact that $e^{-j}\leqslant R$ for every $j\geqslant j_1$, we have that 
\begin{equation*}
\Hh^s_\infty(F\cap B)\leqslant2^s\sum_{j\geqslant j_1}C_2e^{jd}(R+2e^{-j})e^{-j(d+d/\sqrt{j})}\leqslant 2^s3C_2R\sum_{j\geqslant j_1}e^{-d\sqrt{j}}.
\end{equation*}
Taking the limit when $j_1\to\infty$ we conclude that $\Hh^s_\infty(F\cap B)=0$, and thus, by $(i)$ of Theorem \ref{teoCAl}, $F\notin\Jj^s$.
\end{example}

A third definition of these type of sets was given by Bugeaud in \cite{B}. The idea is to generalize \eqref{Bugeaud} considering a large class of {\it dimension functions}. A function $f:[0,+\infty)\to[0,+\infty)$ is called a dimension function if it is strictly increasing, continuous and $f(0)=0$. 

If $f$ is a dimension function we denote $\varepsilon(f)$ \footnote{Bugeaud's definition is slightly more demanding: $f$ must be concave in $[0,x]$ for all $x<\varepsilon(f)$. This, however, does not affect the definition or the results that follow, since the key property still is sub-additivity.} as the supremum of all positive numbers $x$ such that the function defined as $t\mapsto f(t^{1/d})$ is sub-additive on $[0,x]$. Then for a set $E\subset\R^d$ we define
\begin{equation*}
\Mm_\infty^f(E)=\inf\left\{\sum_nf(\diam(Q_n)):E\subset\bigcup_nQ_n, Q_n\text{ dyadic cube with }\diam(Q_n)\leqslant\varepsilon(f)\right\},
\end{equation*}
were, as in the introduction, we are considering the dyadic cubes in $\R^d$ as the family $Q_{j,k}=2^j(k+[0,1)^d)$, $k\in\Z^d$, $j\in\Z$ (or $\N$). As before, $\Mm_\infty^f$ is equivalent to the Hausdorff content $\Hh^f_\infty$ defined similarly but with any type of covering. Subbaditivity plays again a key role in the identity
\begin{equation*}
\Mm_\infty^f(Q)=f(\diam(Q))
\end{equation*}
for all cube $Q$ with $\diam(Q)\leqslant\varepsilon(f)$.

The last ingredient to define $\Gg^f$-sets is the relation of dimension functions: if $g$ is another dimension function we write $g\prec f$ when 
\begin{equation*}
g(x)/f(x)\to+\infty
\end{equation*}
monotonically when $x$ decreases to $0$ on a neighborhood of $0$. Finally we say that a $G_\delta$ subset $F\subset\R^d$ is a $\Gg^f$-set if
\begin{equation}\label{GsBugeaud}
\Mm^g_\infty(F\cap Q)=\Mm^g_\infty(Q)    
\end{equation}
holds for every dyadic cube $Q$ and every dimension function $g\prec f$. Equivalent characterizations as in Theorem \ref{teoB} also exists for the class $\Gg^f$, see \cite[Theorem 6]{B}.

One could naturally extend the definition of the class $\Gg^f$ to a metric space $X$ with a doubling measure $\mu$ by substituting $\diam{(Q_n)}$ by $\mu(Q_n)^{1/d}$ where $d=\dim_\Hh X$. In a similar way we can consider $\Jj^f_\mu$-sets for $\R^d$ or more general metric spaces by changing \eqref{GsBugeaud} to the same identity for just $f$. In which case we have
\begin{equation*}
\Gg^f_\mu(X)=\bigcap_{g\prec f}\Jj^g_\mu(X).
\end{equation*}
Moreover, it can be shown \cite[Lemma 2]{B} that $\Jj^f\subset\Gg^f$. On the other hand, when we consider Falconer's definition and the dimension function $f(x)=x^s$ we have $\Gg^{x^s}\subset\Gg^s$ since $x^t\prec x^s$ whenever $t<s$. Altogether this means that
\begin{equation*}
\Jj^s=\Jj^{x^s}\subset\Gg^{x^s}\subset\Gg^s
\end{equation*}
for $0<s\leqslant d$. And we know from Example \ref{ejemploreal} that $\Jj^s\varsubsetneq\Gg^s$. We will now show that Bugeaud's definition coincides with our definition of $\Jj^s$ in $\R^d$.

\begin{proposition}\label{GsJs}
For all $0<s\leqslant d$ we have $\Gg^{x^s}=\Jj^s$.
\end{proposition}

\begin{proof}
Since $\Jj^s\subset\Gg^{x^s}$, we only need to prove the converse inclusion. To that end we will show that if $F\notin\Jj^s$, then there exists $g\prec x^s$ such that \eqref{GsBugeaud} is not verified for $g$ and $F$.

The condition $F\notin\Jj^s$ means that there exists a dyadic cube $Q$ and a constant $c<1$ for which we can find a dyadic covering $\{Q_n\}$ of $F\cap Q$ such that
\begin{equation}\label{noC}
\sum_n\Ll(Q_n)^{s/d}<c\Ll(Q)^{s/d},    
\end{equation}
where, as before, $\Ll$ denotes the Lebesgue measure in $\R^d$. 

Since $\diam(Q)=\sqrt{d}\Ll(Q)^{1/d}$, then \eqref{noC} translates to
\begin{equation*}
\sum_n\diam(Q_n)^s<c\diam(Q)^s.
\end{equation*}
From here we will proceed to construct a dimension function $g\prec x^s$. First, we denote $x_n=\diam(Q_n)$ and $a_n=x_n^s$. Eventually reordering the cubes we may assume that the sequence $\{a_n\}$ is decreasing. Remember that, according to the cubes we consider in $\R^d$, for all $n$ we must have $x_n=\sqrt{d}2^{m_n}$ for some $m_n\in\Z$. Thus, in particular we have that, for all $n$, either $x_{n+1}=x_n$ or $x_{n+1}\leqslant x_n/2$.

For every $n$ we set
\begin{equation*}
A_n=\sum_{m\geqslant n}a_m.
\end{equation*}
Then for every $k$ we consider $\nu(k)$ such that $A_{\nu(k)}<2^{-k}$ and $a_{\nu(k)}<a_{\nu(k)-1}$. The second condition implies $a_{\nu(k)}/a_{\nu(k)-1}\leqslant2^{-s}$. We also take $k_0$ such that for every $k\geqslant k_0$, 
\begin{equation}\label{condicionh}
1+\frac{1}{1+k}<2^{s},    
\end{equation}
and for every $r\in(0,1)$ we define the function $h_r:\N\to\N$ by
\begin{equation*}
h_r(k)=\left\{\begin{array}{cc}
1+rk_0&\text{ if }k\leqslant k_0,
\\
1+rk&\text{ if }k> k_0.
\end{array}\right.
\end{equation*}
Next, we construct a sequence $\{b_n\}$ in the following way: if $n\in [\nu(k),\nu(k+1))$, we set
\begin{equation*}
b_n:=h_r(k)a_n.
\end{equation*}

Then we have $b_n/a_n\nearrow+\infty$ and
\begin{equation*}
\sum_nb_n=\sum_{k}\sum_{\nu(k)}^{\nu(k+1)-1} h_r(k)a_n\leqslant\sum_k h_r(k)A_{\nu(k)}\leqslant\sum_kh_r(k)2^{-k}<+\infty.
\end{equation*}
Also observe that either $b_{n+1}/b_n=a_{n+1}/a_n\leqslant1$ or
\begin{equation*}
\frac{b_{n+1}}{b_n}\leqslant\frac{h_r(k)a_{n+1}}{h_r(k-1)a_n}\leqslant\frac{h_r(k)a_{\nu(k)}}{h_r(k-1)a_{\nu(k)-1}}<1,
\end{equation*}
where the last inequality comes from the condition $a_{\nu(k)}/a_{\nu(k)-1}\leqslant2^{-s}$ together with \eqref{condicionh} and the fact that for every $k$ the function $r\mapsto r/(1+rk)$ is increasing. We conclude that $b_n$ is decreasing.

Now we define $H_r:(0,+\infty)\to(0,+\infty)$ such that:
\begin{itemize}
\item $H_r(x_n)=b_n/a_n$,
\item $H_r$ is affine on each interval $[x_{n+1},x_n]$, and
\item $H_r(x)=H_r(x_0)$ for every $x\geqslant x_0$.
\end{itemize}
This function is clearly decreasing and satisfies
\begin{equation}\label{prec}
\lim_{x\to 0}H_r(x)=+\infty.
\end{equation}

Finally we set $g_r:[0,+\infty)\to[0,+\infty)$ as
\begin{equation*}
g_r(x):=\left\{\begin{array}{ll}
H_r(x)x^s&\text{if }x>0,
\\
0&\text{if }x=0.
\end{array}\right.
\end{equation*}
We claim that $g_r$ is a dimension function. To see this it is enough to show that $g_r$ is increasing on every interval $[x_{n+1},x_n]$. Indeed, since $g_r$ is continuous in $(0,+\infty)$, $g_r(x_n)=b_n$ and $b_n\searrow0$, then, if $g_r$ is increasing on every interval $[x_{n+1},x_n]$, $g_r$ is global increasing and continuous at $0$.

If $H_r(x_n)=H_r(x_{n+1})$, then $g_r'(x)=sH_r(x_n)x^{s-1}>0$ for every $x\in(x_{n+1},x_n)$. In the other case $H_r(x_n)=h_r(k)<h_r(k+1)=H_r(x_{n+1})$ with $k\geqslant k_0$, which implies
\begin{equation*}
H_r(x)=\left(\frac{-r}{x_n-x_{n+1}}\right)x+\frac{rx_n}{x_n-x_{n+1}}+1+rk,\qquad\text{for every }x\in(x_{n+1},x_n).
\end{equation*}
Then for every $x\in(x_{n+1},x_n)$,
\begin{align*}
g_r'(x)=H_r'(x)x^s+sH_r(x)x^{s-1}&=x^{s-1}\left(\frac{-rx}{x_n-x_{n+1}}-\frac{srx}{x_n-x_{n+1}}+\frac{srx_n}{x_n-x_{n+1}}+s+srk\right)
\\
&>x^{s-1}r\left(\frac{-(1+s)x_n}{x_n-x_{n+1}}+sk\right)\geqslant x^{s-1}r\left(-2(1+s)+sk\right),
\end{align*}
where the last inequality comes from the relation $x_{n+1}\leqslant x_n/2$. Thus, if we choose $k_0\geqslant 2(1+s)/s$, then $g_r'(x)>0$ for every $x\in(x_{n+1},x_n)$.

Finally note that, by \eqref{prec}, $g_r(x)\prec x^s$ and
\begin{equation*}
\sum_ng_r(\diam(Q_n))=\sum_ng_r(x_n)=\sum_nb_n<+\infty.
\end{equation*}
Moreover, since $g_r(x)$ decreases pointwise to $x^s$ when $r\searrow0$, the Dominated Convergence Theorem implies that 
\begin{equation*}
\lim_{r\to 0}\sum_ng_r(\diam(Q_n))=\sum_n\diam(Q_n)^s<c\diam(Q)^s=c\lim_{r\to 0}g_r(\diam(Q)).
\end{equation*}
Which means that there exists $r_0\in(0,1)$ small enough such that if $g=g_{r_0}$, then 
\begin{equation}\label{final}
\Mm^g_\infty(F\cap Q)\leqslant\sum_ng(\diam(Q_n))<cg(\diam(Q))=c\Mm^{g}_\infty(Q),
\end{equation}
where the last equality comes from the sub-additivity of $g(x^{1/d})$: as $\displaystyle\frac{g(x)}{x^s}$ is decreasing then given $0\leqslant x\leqslant y$
\begin{align*}
g((x+y)^{1/d})\leqslant(x+y)^{s/d}\frac{g(y^{1/d})}{y^{s/d}}&\leqslant(x^{s/d}+y^{s/d})\frac{g(y^{1/d})}{y^{s/d}}
\\
&\leqslant g(y^{1/d})+x^{s/d}\frac{g(x^{1/d})}{x^{s/d}}\leqslant g(x^{1/d})+g(y^{1/d})
\end{align*}
also using the restriction $0<s\leqslant d$.

Anyhow, \eqref{final} shows that $F$ does not verify \eqref{GsBugeaud} for $g$ and $Q$ and therefore $F\notin\Gg^{x^s}$.
\end{proof}

As a consequence of Example \ref{ejemploreal} and Proposition \ref{GsJs} we conclude that Falconer's and Bugeaud's definitions give different families. The key difference between them is the fact that the dimension functions $x\mapsto x^t$, with $t<s$, are not ``dense enough'' within the family of all dimension functions $g\prec x^s$. A more direct way to see this is by considering the following example.

\begin{example}\label{BvsF}
Fix $\alpha>1$ and for each $j$ let $\{x_{j,k}\}_k$ be a $(1,2,e^{-j})$-net in $\R^d$. For every $j$ we define
\begin{equation*}
E_j=\bigcup_{k} B\left(x_{j,k},\frac{e^{-\alpha j}}{j^\kappa}\right),
\end{equation*}
where $\kappa>0$ is to be chosen later. Finally, we set $F:=\limsup_j E_j$.
We observe that for every $\delta>\alpha$ there exists $j_{\delta,\kappa}$ such that $\frac{e^{-\alpha j}}{j^\kappa}\geqslant e^{-\delta j}$ for every $j\geqslant j_{\delta,\kappa}$, and thus $F$ contains a $\delta$-approximable set for every $\delta>\alpha$. Then, by Theorem \ref{teoD}, we have $F\in\Gg^{d/\delta}$ for every $\delta>\alpha$, which implies $F\in\Gg^{d/\alpha}$.

We consider the dimension functions $f(t)=t^{d/\alpha}$ and 
\begin{equation*}
g(t)=t^{d/\alpha}\log(1+t^{-1})^\beta.
\end{equation*}
with $\beta>0$, which clearly satisfy $g\prec f$. Proceeding as in Example \ref{ejemploreal} and choosing $\kappa$ and $\beta$ so that $d\kappa/\alpha-\beta>1$ we obtain $\Hh^g_\infty(F)=0$, which implies $\Mm^g_\infty(F)=0$ and therefore $F\notin\Gg^f$.
\end{example}

In Examples \ref{ejemploreal} and \ref{BvsF} we have slightly modified the definion of $\alpha$-approximable sets in order to construct subsets that belong to certain classes and not to others. Let us use this idea once again to construct the following example, where a kind of $\Jj$-condition (or $\Gg$-condition) is satisfied locally.

\begin{example}
Consider $X=[2,\infty)$ and for every $j\in\N$ let $\{x_{j,k}\}_k$ be a family of $(1,2,e^{-j})$-nets. Now we define for each $j$
\begin{equation*}
E_j:=\bigcup_k B\left(x_{j,k},e^{-jx_{j,k}}\right),
\end{equation*}
and then $F=\limsup E_j$.

For every closed interval $\mathcal{I}=[a,b]\subset X$ we denote $\mathcal{I}_j=\{k : x_{j,k}\in \mathcal{I}\}$ and can consider two subsets $F^{\mathcal{I}}=\limsup E_j^{\mathcal{I}}$ and $F_{\mathcal{I}}=\limsup E^j_{\mathcal{I}}$ with
\begin{equation*}
E_j^{\mathcal{I}}=\bigcup_{k\in \mathcal{I}_j}B(x_{j,k},e^{-ja})\text{ and }E^j_{\mathcal{I}}=\bigcup_{k\in \mathcal{I}_j}B(x_{j,k},e^{-jb}).
\end{equation*}
It is easy to see that $F_{\mathcal{I}}\subset F\cap I\subset F^{\mathcal{I}}$. Then $\frac{1}{b}\leqslant \dim_\Hh F\leqslant \frac{1}{a}$ because $F_\mathcal{I}$ and $F^\mathcal{I}$ are sets of $b$ and $a$-approximable points (respectively) and $\mathcal{E}$ has exponential decay. Since $F\cap \mathcal{I}$ has subsets of dimension arbitrarily close to $\frac{1}{a}$, we conclude that $\dim_\Hh F\cap \mathcal{I}=\frac{1}{a}$. The same argument shows that $\dim_\Hh F =\frac{1}{2}$. On the other hand, we have that $F\cap \mathcal{I}\in\Gg^{1/b}(\mathcal{I})$. 

This is a very irregular metric space. For example it is clear that its only self locally bi-Lipschitz map is the identity.
\end{example}

\section*{Acknowledgments}
Part of this work was done while both authors were working at Centro de Matemática in Universidad de la República, Uruguay.  The first author was financed by a 3 month postdoctoral grant of the same institution.

The second author is supported by the Mathematical Center in Akademgorodok under the agreement No. 075-15-2019-1675 with the Ministry of Science and Higher Education of the Russian Federation.

\end{document}